\documentclass[]{article}
\usepackage{amsmath,amsthm,amsfonts}
\usepackage[all]{xy}
\SelectTips{cm}{}

\newtheorem{theorem}{Theorem}[section]
\newtheorem{lemma}[theorem]{Lemma}

\newcommand{\xymap}[3]{\xymatrix@1{
#1 \ar[r]^-{_{#3}}  & #2  }
}
\newcommand{\minibracket}[1]{{\text{\small (}}#1{\text{\small )}}}

\newcommand{\cat}[1]{\text{\textup{\textsf{#1}}}}

\newcommand{\fpmod}[1]{\cat{fpmod}(#1)}

\newcommand{\mmod}{\cat{mod}}  
\newcommand{\tses}[3]{0\rightarrow #1\rightarrow #2 \rightarrow%
#3\rightarrow 0}%
\newcommand{\shortexactsequence}[3]{\xymatrix@1{
0 \ar[r]  & #1 \ar[r]  & #2 \ar[r]  & #3 \ar[r]  & 0  }
}
\newcommand{\map}[3]{#1 \colon #2 \to #3}
\newcommand{\tensor}{\otimes}
\newcommand{\N}{\mathbb{N}}

\newcommand{\Mod}[1]{\cat{Mod}(#1)}

\newcommand{\textmap}[2]{\left(\begin{smallmatrix}
#1\\ #2
\end{smallmatrix}
\right)}
\newcommand{\define}[1]{\textit{#1}}

\newcommand{\pullses}[5]{%
\xymatrix{%
0 \ar[r]  & #1 \ar[r] \ar@<0.5ex>@{-}[d] \ar@<-0.5ex> @{-}[d]  & #2%
\ar[r] \ar[d] & #3 \ar[r] \ar[d] & 0\\%
0 \ar[r]  & #1 \ar[r]  & #4 \ar[r]  & #5 \ar[r]  & 0  }%
}%

\newcommand{\pullsesalong}[6]{%
\xymatrix{%
0 \ar[r]  & #1 \ar[r] \ar@<0.5ex>@{-}[d] \ar@<-0.5ex> @{-}[d]  & #2%
\ar[r] \ar[d] & #3 \ar[r] \ar[d]^{#6} & 0\\%
0 \ar[r]  & #1 \ar[r]  & #4 \ar[r]  & #5 \ar[r]  & 0  }%
}%

\newcommand{\pushses}[5]{%
\xymatrix{%
0 \ar[r]  & #1 \ar[r] \ar[d]  & #2 \ar[r] \ar[d] & #3 \ar[r] \ar@<0.5ex>%
@{-}[d]  \ar@<-0.5ex> @{-}[d]   & 0\\%
0 \ar[r]  & #4 \ar[r]  & #5 \ar[r]  & #3 \ar[r]  & 0  }%
}%

\newcommand{\pushsesalong}[6]{%
\xymatrix{%
0 \ar[r]  & #1 \ar[r] \ar[d]^{#6}  & #2 \ar[r] \ar[d] & #3 \ar[r] \ar@<0.5ex>%
@{-}[d]  \ar@<-0.5ex> @{-}[d]   & 0\\%
0 \ar[r]  & #4 \ar[r]  & #5 \ar[r]  & #3 \ar[r]  & 0  }%
}%

\newcommand{\Qnloops}{%
\underset{_{n \text{ arrows}}}{%
\xymatrix{%
\cdot \ar@(dl,ul)[] \ar@(ur,dr)[] \ar@(ul,ur)@{.>}[]%
}}}%
\newcommand{\QnKronecker}{%
\underset{_{n \text{ arrows}}}{\xymatrix{%
\cdot \ar@<1.2ex>[r] \ar@<-1.2ex>[r] \ar@<0.9ex>@{}[r] |{\cdot}%
\ar@<0.3ex>@{}[r] |{\cdot} \ar@<-0.3ex>@{}[r] |{\cdot}%
\ar@<-0.9ex>@{}[r] |{\cdot} &\cdot%
}}%
}%
\newcommand{\QnA}{%
\underset{_{n \text{ vertices}}}{\xymatrix{%
\cdot \ar[r] &\cdot\  \ar@{}[r] |{\cdots\cdots\cdots} &\ \cdot \ar[r] &\cdot%
}}%
}%
\newcommand{\commsquare}[8]{\xymatrix{%
#1 \ar[r]^-{#5} \ar[d]_-{#6}  & #2 \ar[d]^-{#7} \\%
#3 \ar[r]_-{#8}  & #4  }%
}%

\DeclareMathOperator{\Hom}{Hom}
\DeclareMathOperator{\End}{End}
\DeclareMathOperator{\im}{im}
\DeclareMathOperator{\cok}{cok}
\DeclareMathOperator{\Ext}{Ext}
\DeclareMathOperator{\Tor}{Tor}

\begin{document}

\title{Severe right Ore sets and universal localisation}
\author{Aidan Schof\i eld}
\maketitle

\begin{abstract} 
We introduce the notion of a severe right Ore set in the main as a
tool to study universal localisations of rings but also to provide a
short proof of P. M. Cohn's classification of homomorphisms from a
ring to a division ring. We prove that the category of finitely
presented modules over a universal localisation is equivalent to a
localisation at a severe right Ore set of the category of finitely
presented modules over the original ring. This allows us to describe
the structure of finitely presented modules over the universal
localisation as modules over the original ring. 
\end{abstract} 

\section{Introduction}
 
The main purpose of this paper is to introduce a type of right Ore
set in an additive category with cokernels and to demonstrate the use
of this notion in two ways. The first is to provide a short proof of
P. M. Cohn's characterisation of epimorphisms from a given ring to
division rings; the second is to study universal localisation. 

The main theorem we prove about universal localisation is that the
category of finitely presented modules over a universal localisation
is a right Ore localisation of the category of finitely presented
modules over the original ring at the severe right Ore generated by
the maps between finitely generated projective modules we wish to
invert (see section \ref{sec:sevrO} for the definition of a severe
right Ore set). To some extent, this result is a surprise since
another approach to the study of universal localisation would be to
study the derived category of the universal localisation and here a
corresponding result fails to be true; the derived category of the
universal localisation can fail to be the right perpendicular category
to the maps between finitely generated projective modules considered
as objects in the derived category (see \cite{RanNeeman}).

This allows us to give a module-theoretic description of the finitely
presented modules over the universal localisation as modules over the
original ring. From this we can give a description of the kernel of
the homomorphism from a module to the induced module over the
universal localisation. Although this answer is useful, there are many
situations where we want a simpler condition. Specifically we should
like to be able to say that this kernel is simply the torsion
submodule with respect to the torsion theory generated by the
cokernels of the maps between finitely generated projective modules we
invert. This is false in general; however, we provide a simple and
fairly general sufficient condition on the universal localisation for
this to hold. When this does hold we can provide detailed information
about the universal localisation so we investigate these particular
universal localisations further in the final section. In particular,
we show that for these universal localisations
$\Tor_{i}^{R}(R_{\Sigma}, R_{\Sigma}) $ vanishes for $i>0$ which is
the condition required in \cite{RanNeeman} so that the derived
category of the universal localisation should be the right
perpendicular category to the maps between finitely generated
projective modules considered as objects in the derived category and
hence to construct a long exact sequence for universal localisation in
algebraic $K$-theory as demonstrated in \cite{RanNeeman}.  

In a subsequent paper we shall use these results to describe the
universal localisations of hereditary rings very precisely. We can
find all possible universal localisations in terms of suitable
subcategories of the category of finitely presented modules over the
original ring; we can then describe the category of finitely presented
bound modules over the universal localisation as being equivalent to a
suitable subcategory of finitely presented bound modules over the
original ring and we can describe the finitely generated projective
modules over the universal localisation in terms of the submodules of
the cokernels of the maps we invert.

\section{Severe right Ore sets} 
\label{sec:sevrO} 

We recall the definition of a right Ore set in a small additive
category $\cat{A}$.  A set of maps $\sigma$ in a small additive
category is said to be a right Ore set if the following conditions are
satisfied:
\begin{enumerate}
\item The set $\sigma$ contains all isomorphisms in $\cat{A}$. 
\item \define{(Closed under composition).} If $s,t\in \sigma$ and
$st$ exists then $st\in \sigma$.
\item \define{(Common right multiples exist).} If $s \colon A
\rightarrow B$ lies in $\sigma$ and $a \colon A \rightarrow C$ is some
map in the category, there exist $t\in \sigma$ and a map $b$ such that
$at=sb$. 
\item \define{($\sigma$ is right revengeful).} Suppose $s\in
\sigma$ and $a$ is a map such that $sa=0$; then there exists $t\in
\sigma$ such that $at=0$. 
\end{enumerate}

We say that two maps $a,b$ in a category are \define{isomorphic}
if there exist isomorphisms $u,v$ such that $b=uav$.  

Although it is certainly possible in the definition of a right Ore set
to get by without the first assumption and to modify the second
assumption to the condition that $st$ is isomorphic to a map in
$\sigma$, we gain nothing by doing so.

Given a right Ore set in an additive category $\cat{A}$, we are able
to describe the maps in the category $\cat{A}_{\sigma}$ very
precisely. Every map may be written in the form $as^{-1}$ and
$as^{-1}=bt^{-1}$ if and only if there exist maps $u,v\in \sigma$ such
that $au=bv$ and $su=tv$; in particular, $as^{-1}=0$ if and only if
there exists $u\in \sigma$ such that $au=0$. We leave it to the reader
to check that this is true or to consult \cite{GabZis}.
%%% Gabriel-Zisman. 

We use the notation $[A]$ for the image of $A$ in the
category $\cat{A}_{\sigma}$ where $\sigma$ is a right Ore set in
$\cat{A} $.  

Before stating the next theorem, we should point out that we shall be
working with additive categories with cokernels as our standard
type of category for much of this paper. The reason for this is that
the category of finitely presented modules over a ring $R$ is of this type
and this category is only an abelian category if the ring is
coherent. We shall use the notation $\fpmod{R}$ for this category of
finitely presented modules over the ring $R$. We may and will regard
$\fpmod{R}$ as a small category.

Given a functor between additive categories with cokernels
we say that this functor is right exact if and only if it preserves
cokernels. We shall occasionally say that a sequence $A\to B\to C\to
0$ is exact by which we mean simply that the map from $B$ to $C$ is
the cokernel of the map from $A$ to $B$.
 
We note the following standard lemma which is usually
stated in the context of the full module category over the ring $R$
and a right exact functor to an abelian category but whose proof is
identical and obvious in this context. 

\begin{lemma}
\label{lem:charrex}
Let $\map{\phi}{\fpmod{R}}{\cat{A}}$ be a right exact functor where
$\cat{A}$ is an additive category with cokernels. Then $\phi$ is
naturally equivalent to $\tensor_{R} \phi(R) $.
\end{lemma}

\begin{theorem}
\label{th:Olocrex}
Let $\cat{A}$ be an additive category with cokernels and let $\sigma$
be a right Ore set. Then $\cat{A}_{\sigma}$ is an additive
category with cokernels and the functor from $\cat{A}$ to
$\cat{A}_{\sigma}$ is right exact. 
\end{theorem}
\begin{proof}
Let $b \colon B \rightarrow C$ be the cokernel of $a \colon A
\rightarrow B$. Let $s \colon D \rightarrow B$ lie in $\sigma$. We
consider the map $sb$ in $\cat{A}_{\sigma}$ which we hope is the
cokernel of $as^{-1}$.  Let $ct^{-1}$ be a map such that
$as^{-1}ct^{-1}=0$. There exist $c'\in \cat{A}$ and $s'\in \sigma$
such that $sc'=cs'$ in $\cat{A}$ and so $s^{-1}c = c's'^{-1}$ and
$0=ac'(ts')^{-1}$. Hence there exists $u\in \sigma$ such that $ac'u=0$
in $\cat{A}$. So $c'u = bd$ for some map $d$ in $\cat{A}$. Therefore
$ct^{-1}=sc's'^{-1}t^{-1} = sc'u (ts'u)^{-1} = sbd(ts'u)^{-1}$ as
needed. Hence $sb$ is the cokernel of $as^{-1}$. Thus
$\cat{A}_{\sigma}$ is an additive category with cokernels and the
functor from $\cat{A}$ to $\cat{A}_{\sigma}$ is right exact since the
cokernel of $a=a1^{-1}$ is still $b$.
\end{proof}

Now suppose that $\cat{A}$ is an additive category with cokernels. 
Let $a \colon A \rightarrow B$ and $b \colon A \rightarrow C$ be maps
in $\cat{A}$. Let $\left(\begin{smallmatrix}
b'\\
a'
\end{smallmatrix}\right)\colon B\oplus C\rightarrow \cok (a\ b)  
$ be the canonical map to the cokernel; then we call $a'$ the
\define{pushout} of $a$ along $b$; similarly, $b'$ is the pushout
of $b$ along $a$. By definition, $ab'=-ba'$.  

Let $\sigma$ be a set of maps in the small category $\cat{A}$. We say
that $\sigma$ is a severe right Ore set if it satisfies the following
conditions:
\begin{enumerate}
\item The set $\sigma$ contains all isomorphisms.
\item It is closed under composition.
\item \define{(Closed under pushout).} Given $s \colon A
\rightarrow B$ in $\sigma$ and $a \colon A \rightarrow C$, the pushout
of $s$ along $a$ lies in $\sigma$.
\item \define{(Severely right revengeful or closed under cokernels
of right killers).} If $s \colon A \rightarrow B$ is in $\sigma$ and
$a \colon B \rightarrow C$ is a map such that $sa=0$ then $t \colon C
\rightarrow \cok a $, the canonical map to the cokernel, also lies in
$\sigma$.   
\end{enumerate}

The last two conditions imply directly the corresponding conditions
for a right Ore set so that the name is justified. There are two good
reasons for considering this definition. Firstly, because a severe
right Ore set is defined by closure conditions, it is possible to talk
of the severe right Ore set generated by a set of maps. The second
reason is the following lemma.

\begin{lemma}
\label{th:lsroc}
Let $\phi\colon R\rightarrow S$ be a ring homomorphism. Let $\sigma$
be the set of maps $\{s\in\fpmod{R}: s\otimes_{R}S \text{ is an
isomorphism}\}$. Then $\sigma$ is a severe right Ore set.
\end{lemma}
\begin{proof}
It is clearly closed under composition and contains all
isomorphisms. Now suppose that $s \colon A \rightarrow B$ lies in
$\sigma$. Let $a \colon A \rightarrow C$ be some map. Let
$f=\left(\begin{smallmatrix} a'\\ s'
\end{smallmatrix}
 \right)\colon B\oplus C\rightarrow D$ be the cokernel of $(s\ a)$. Then,
since $\_\otimes_{R}S$ is right exact, $f\otimes_{R}S$ is still the
cokernel of $(s\ a)\otimes_{R}S$ and so $s'\otimes_{R}S$ is the
pushout of $s\otimes_{R}S$ along $a\otimes_{R}S$. However the pushout
of an isomorphism is an isomorphism. So $\sigma$ is closed under
pushout. 

Now suppose that $s \colon A \rightarrow B$ is in $\sigma$ and $a
\colon B \rightarrow C$ is a map such that $sa=0$. Let $t \colon C
\rightarrow D$ be the cokernel of $a$. Since $s\otimes_{R}S$ is an
isomorphism, $a\otimes_{R}S =0$. Since $\_\otimes_{R}S$ is right exact,
$t\otimes_{R}S$ is the cokernel of $a\otimes_{R}S =0$ and so is an
isomorphism. Thus $\sigma$ is severely right revengeful as well and
hence is a severe right Ore set. 
\end{proof}

The reader may check that the proof above actually shows that for any
right exact functor between small additive category with cokernels the
set of maps inverted by the functor must be a severe right Ore set. 

Let $\Sigma$ be a set of maps between finitely generated projective
modules over the ring $R$. Then we can define $\sigma$ to be the
severe right Ore set generated by $\Sigma$ in $\fpmod{R}$. We can then
form the category $\fpmod{R}_{\sigma}$. Our main aim in this paper is to
show that this category is naturally equivalent to
$\fpmod{R_{\Sigma}}$ via the natural functor induced by
$\otimes_{R}R_{\Sigma}\colon \fpmod{R}\rightarrow \fpmod{R_{\Sigma}}$
which inverts $\Sigma$ and so by lemma \ref{th:lsroc} also inverts
$\sigma$. We return to this kind of severe right Ore set later but for
the moment we return to the general case.

Let $\sigma$ be a severe right Ore set in the category $\fpmod{R}$ we
define a functor $\Gamma$ from $\fpmod{R}_{\sigma}$ to $\Mod{R}$
which we shall refer to as the \define{realisation functor} by
$\Gamma=\Hom_{\fpmod{R}_{\sigma}}([R],\ )$. We shall see that the
realisation functor is a full and faithful functor. We begin by
describing the image of $[M]$ under $\Gamma$. Let ${_{M}\sigma}$ be
the set of maps in $\sigma$ whose domain is $M$. We have a directed
system of modules $\{M_{s}\}$ indexed by ${_{M}\sigma}$ where a morphism
from $M_{s}$ to $M_{u}$ is given by $t \colon M_{s} \rightarrow M_{u}$
where $t\in \sigma$ and $u=st$. There is an initial object $M=M_{1}$
since $1_{M}\in \sigma$. Moreover given $s,t\in {_{M}\sigma}$, the
pushout diagram
\begin{equation*}
\commsquare{M}{M_{s}}{M_{t}}{M_{u}}{s}{t}{s'}{t'} 
\end{equation*}
where $u=st'=s't$ lies in our system because $\sigma$ is closed under
pushout and composition. We define $M_{\sigma}$ to be the direct
limit of this system $\varinjlim_{s\in {_{M}\sigma}}M_s$.

Let $M$ and $N$ be finitely presented modules. Because $M$ is finitely
presented $\Hom(M,N_{\sigma})= \varinjlim_{t\in
{_{N}\sigma}}\Hom(M,N_{t})$. We define $\lambda_{t}\colon
\Hom(M,N_{t})\rightarrow \Hom([M],[N])$ by
$\lambda_{t}(f)=ft^{-1}$. Given a map in $\sigma$, $u\colon
N_{t}\rightarrow N_{tu}$, we see that $(M,u)\lambda_{tu}=\lambda_{t}$
since $\minibracket{f}(M,u)\lambda_{tu} = fu(tu)^{-1} = ft^{-1}$. Thus
we obtain a map $\lambda\colon \varinjlim_{t\in {_{N}\sigma}}\Hom(M,N_{t})
\rightarrow \Hom([M],[N])$ which is visibly surjective.

\begin{lemma}
\label{th:limhom=hom}
Let $M$ and $N$ be finitely presented modules. Then the map
$\lambda\colon \varinjlim_{t\in {_{N}\sigma}}\Hom(M,N_{t})
\rightarrow \Hom([M],[N])$ is an isomorphism. So
$\Hom(M,N_{\sigma})\cong \Hom([M],[N])$. 
\end{lemma}
\begin{proof}
Suppose that $\lambda(k)=0$. Let $f\in \Hom(M,N_{t})$ represent
$k$. So $ft^{-1}=0$ in $\fpmod{R}_{\sigma}$. Then there exists $u\in
\sigma$ such that $fu=0$ but then $\minibracket{f}(M,u)=fu=0$ and so
$k=0$. Thus $\lambda$ is injective and hence bijective.
\end{proof}

Our first use of this is to identify the images of objects under
$\Gamma$.  

\begin{lemma}
\label{th:realis=dirlim}
Let $\Gamma\colon \fpmod{R}_{\sigma}\rightarrow \Mod{R}$ be the
realisation functor. Then $\Gamma([M])\cong M_{\sigma}$.
\end{lemma}
\begin{proof}
By the previous lemma we see that $\Gamma(M) =\Hom([R],[M]) =
\Hom(R,M_{\sigma}) = M_{\sigma}$. 
\end{proof}

We denote the map from $M_{s}$ to $M_{\sigma}$ by $\iota_{s}$. In
particular $\iota_{1}$ is the natural map from $M$ to $M_{\sigma}$. 

\begin{lemma}
\label{th:homdirlims}
Let $M$ and $N$ be finitely presented modules. Then the homomorphism
$\iota_{1}\colon M\rightarrow M_{\sigma}$ induces an isomorphism
$(\iota_{1},N_{\sigma})\colon\Hom(M_{\sigma},N_{\sigma}) \cong
\Hom(M,N_{\sigma})$.
\end{lemma}
\begin{proof}
Let $f \colon M_{\sigma} \rightarrow N_{\sigma}$ be a homomorphism
such that $\iota_{1}f=0$. Consider the map $\iota_{s}f\colon
M_{s}\rightarrow N_{\sigma}$. Let $c_{s}\colon M_{s}\rightarrow T_{s}$
be the cokernel of $s$. So $\iota_{s}f=c_{s}f'$ for some map $f'$ from
$T_{s}$ to $N_{\sigma}$. Then since $T_{s}$ is finitely presented, we
can choose $t$ so that $f'$ factors through $N_{t}$; thus
$f'=f_{1}\iota_{t}$. Now $c_{s}f_{1}$ is a homomorphism from $M_{s}$
to $N_{t}$ such that $sc_{s}f_{1}=0$ and since $\sigma$ is a severe
right Ore set, the cokernel of $c_{s}f_{1}=0$, $u\colon
N_{t}\rightarrow N_{tu}$ lies in $\sigma$. So $\iota_{t} =
u\iota_{tu}$. Hence
\begin{equation*}
\iota_{s}f = c_{s}f' = c_{s}f_{1}\iota_{t} = c_{s}f_{1}u\iota_{tu} = 0
\end{equation*}
and since this is true for all $s$ we deduce that $f=0$. So
$(\iota_{1},N_{\sigma})$ is injective.

Now let $\phi\colon M\rightarrow N_{\sigma}$ be a homomorphism. Then
since $M$ is finitely presented $\phi$ factors through some
$N_{t}$; that is, $\phi=f_{1}\iota_{t}$; for each $s\in
{_{M}\sigma}$, we consider the pushout diagram:
\begin{equation*}
\commsquare{M}{M_{s}}{N_{t}}{K_{s}}{s}{f_{t}}{f_{s}}{s'} 
\end{equation*}
Then $s'\in \sigma$ and so $K_{s}=N_{ts'}$ and we define
$\phi_{s}\colon M_{s}\rightarrow N_{\sigma}$ to be
$f_{s}\iota_{ts'}$. These maps fit together to give a homomorphism
from $M_{\sigma}$ to $N_{\sigma}$ which restricts to $\phi$ on
$M$ which completes the proof.
\end{proof}

This allows us to conclude that the realisation functor is full and
faithful so that $\fpmod{R}_{\sigma}$ can be thought of as a category
of modules or as we shall find useful later, the category of modules
may be thought of as a right Ore localisation of the category of
finitely presented modules over $R$. 

\begin{theorem}
\label{th:Gammabij}
The functor $\Gamma\colon \fpmod{R}_{\sigma}\rightarrow \Mod{R}$ is
full and faithful.
\end{theorem}
\begin{proof}
We have shown that $\Hom(\Gamma([M]),\Gamma([N]))
= \Hom(M_{\sigma},N_{\sigma}) \cong  \Hom(M,N_{\sigma})$. However,
by lemma \ref{th:limhom=hom}, $\Hom(M,N_{\sigma}) \cong \Hom([M],[N])$
and these isomorphisms are induced by the functor $\Gamma$.  
\end{proof}

We shall return to this line of argument in section \ref{sec:ulOl}. 

\section{Epimorphisms to division rings} 
\label{sec:Cohn}
 
We break from our main thread at this point to prove P. M. Cohn's
characterisation of epimorphisms from a given ring to division
rings. Cohn's characterisation was in terms of prime matrix ideals but
we shall prove a characterisation in terms of Sylvester rank functions
on finitely presented modules. The equivalence of this with prime
matrix ideals is shown in \cite{Malc}. We refer the reader to
\cite{CohnFR} or to \cite{Malc} for the definition of a prime matrix
ideal. 

A Sylvester rank function on an additive category with cokernels
satisfies the following properties. It is a function $\rho$ from
isomorphism classes of objects in $\cat{A}$ to $\N$ which is additive
on direct sums and for every exact sequence $A\to B\to C\to 0$,
$\rho(C) \leq \rho(B) \leq \rho(A)+\rho(C)$. We note that if we have a
right exact functor $\phi$ from $\cat{A}$ to $\mmod(D) $ where $D$ is a
division ring then we have an associated Sylvester rank function on
$\cat{A}$ defined by $\rho(A) =\dim_{D}\phi(A) $. 

We extend the Sylvester rank function to a rank function on maps in
the category by defining the rank of the map $\phi\colon A\to B$ by
the formula $\rho(\phi) = \rho(B) - \rho(\cok(\phi))$. We say that
$\phi$ is $\rho$-full if and only if $\rho(\phi) =\rho(A) =
\rho(B)$. Equivalently, $\phi$ is $\rho$-full if and only if $\rho(A)
= \rho(B)$  and $\rho(\cok(\phi)) =0$.  
    
\begin{lemma}
\label{rkfngivessrO}
Let $\rho$ be a Sylvester rank function on the additive category with
cokernels $\cat{A}$. Then the set of $\rho$-full maps is a severe
right Ore set.
\end{lemma}
 \begin{proof}
   Let $\Pi$ be the set of $\rho$-full maps.
 
   Suppose that $\map{\alpha}{A}{B} $ and $\map{\beta}{B}{C} $ are
   both $\rho$-full maps. Then $\rho(A) = \rho(B)= \rho(C)$ and we
   have an exact sequence $\cok(\alpha)\to \cok(\alpha\beta) \to
   \cok(\beta)\to 0$. Since $\rho(\cok(\alpha))=0=\rho(\cok(\beta))
   $, it follows that $\rho(\cok(\alpha\beta)) \leq
   \rho(\cok(\alpha))+ \rho(\cok(\beta)) = 0$ and so
   $\rho(\cok(\alpha\beta))=0$ and $\alpha\beta$ is $\rho$-full. Thus
   $\Pi$ is closed under composition.

   Suppose that $\map{\alpha}{A}{B}$  is $\rho$-full and
   $\map{\beta}{A}{C}$ is some map. We form the pushout diagram
   \begin{equation*}
     \commsquare{A}{B}{C}{D}{\alpha}{\beta}{\beta'}{\alpha'}        
   \end{equation*}
   By construction we have an exact sequence $A\to B\oplus C\to D\to
   0$ and so $\rho(B)+\rho(C) = \rho(B\oplus C) \leq \rho(A)+\rho(D) $
   and since $\rho(A) = \rho(B)$, it follows that $\rho(C) \leq
   \rho(D)$. On the other hand, the natural map from $\cok(\alpha)$ to
   $\cok(\alpha')$ is surjective so that $\rho(\cok(\alpha'))=0$ and
   so $\rho(D)\leq \rho(\cok(\alpha')) + \rho(C)= \rho(C)$. Thus
   $\alpha'$ is $\rho$-full and $\Pi$ is closed under pushouts. 

   Now assume that $\map{\alpha}{A}{B}$ is $\rho$-full and
   $\map{\beta}{B}{C}$ is a map such that $\alpha\beta=0$. Let
   $\map{\gamma}{C}{D}$ be the cokernel of $\beta$. We need to show
   that $\rho(D)= \rho(C)$ which implies that $\gamma$ is
   $\rho$-full since the cokernel of $\gamma$ is $0$. Since
   $\alpha\beta=0$, $\beta$ induces a map
   $\map{\beta'}{\cok(\alpha)}{C} $ such that $\cok(\alpha)\to C\to
   D\to 0$ is an exact sequence and since $\rho(\cok(\alpha))=0$, it
   follows that $\rho(D)= \rho(C)$ as required. 

   Thus $\Pi$ is severely right revengeful and since we have checked
   all the conditions $\Pi$  is a severe right Ore set. 
 \end{proof}

 We are now in a position to prove the main theorem of this
 section. We show that the right Ore localisation of $\fpmod{R}$ at
 the severe right Ore set constructed in the last theorem must be the
 category of finite dimensional vector spaces over a division ring
 when the rank of the ring itself is $1$.  

\begin{theorem}
  \label{th:Cohnconstrdivrg}
  Let $R$ be a ring and let $\rho$ be a Sylvester rank function on the
  category of finitely presented $R$ modules such that $\rho(R)
  =1$. Let $\Pi$ be the set of $\rho$-full maps in $\fpmod{R} $. Then
  $\fpmod{R_{\Pi}} $ is equivalent to $\mmod(D) $ for some division
  $R$-ring $D$ such that for any finitely presented module $M$,
  $[M]\cong {^{\rho(M)}}D$.
\end{theorem}
\begin{proof}
  First of all, we show that $\End_{\fpmod{R}_{\Pi}}([R]) $ is a
  division ring $D$. Any map from $[R]$ to itself takes the form of
  $\alpha\beta^{-1}$ for maps $\map{\alpha}{R}{M} $ and
  $\map{\beta}{R}{M}$ where $\beta$ is $\rho$-full. It follows that
  $\rho(M)=1 $. So $\rho(\alpha) = 0$ or $1$. If $\rho(\alpha)=1$ then
  $\alpha$ is invertible and so is $\alpha\beta^{-1}$. If
  $\rho(\alpha)=0$ then $\rho(\cok(\alpha))=1$ and the cokernel map is
  $\rho$-full. It follows that $\alpha=0$ in $\fpmod{R}_{\Pi}$ and so
  $\alpha\beta^{-1}=0$ in $\fpmod{R}_{\Pi}$. Thus every endomorphism
  of $[R]$ in $\fpmod{R}_{\Pi}$ is either $0$ or invertible. So
  $\End_{\fpmod{R}_{\Pi}}([R])$ is a division ring $D$ and the natural
  ring homomorphism from $\End_{R}(R) $ to
  $\End_{\fpmod{R}_{\Pi}}([R])$ makes $D$ an $R$-ring. 

  Next we show that if $\rho(M)=m$ then there is a $\rho$-full map
  $\map{\alpha}{{^{m}R}}{M}$ by induction on $m$. If $m=0$, then the
  map from $0$ to $M$ is $\rho$-full. Otherwise, suppose that $m>0$. 
  
  To prove this we need to show that if $\map{\alpha}{R}{M}$ is a map
  such that $\rho(\alpha)=0$ and $\map{\beta}{L}{M}$ is some map then
  $\rho(\beta\oplus \alpha) = \rho(\beta) $. This follows because the
  commutative square
  \begin{equation*}
    \commsquare{M}{\cok(\alpha)}{\cok(\beta)}{\cok(\beta \oplus
      \alpha)}{}{}{}{}
  \end{equation*}
  is a pushout diagram and the map from $M$ to $\cok(\alpha)$ is
  $\rho$-full since it is surjective and $\rho(\alpha)=0$ implies that
  $\rho(M) = \rho(\cok(\alpha))$. It follows that that the map from
  $\cok(\beta)$ to $\cok(\beta \oplus \alpha) $  is $\rho$-full and in
  particular $\rho(\cok(\beta \oplus \alpha)) = \rho(\cok(\beta))$. We
  deduce that $\rho(\beta\oplus \alpha) = \rho(\beta)$. 

  Now suppose that $M$ is generated by the elements $\{m_1,\dots,
  m_t\}$. Let $\map{\phi_i}{R}{M}$ be the map $\phi(r) =m_ir$. If
  $\rho(\phi_i) =0$  for each $i$ then induction and the preceding
  paragraph shows that the surjective map
  $\map{\oplus_{i=1}^{t}\phi_i}{^{t}R}{M}$ has $\rho$-rank $0$ and
  consequently $\rho(M)=0$. The assumption that $\rho(M)>0$ implies
  that we may find some map $\map{\phi}{R}{M}$ such that
  $\rho(\phi)>0$ and hence $\rho(\phi)=1$. Let $M'=\cok(\phi) $. Then
  $\rho(M')=m-1$ and so by induction there exists a $\rho$-full map
  $\map{\mu}{^{m-1}R}{M'}$. We choose some lifting
  $\map{\mu'}{^{m-1}R}{M}$ such that the induced map to $M'$ is
  $\mu$. We consider the map $\map{\phi\oplus\mu'}{^{m}R}{M}$ whose
  cokernel is by construction $\cok(\mu) $. Since $\mu$ is
  $\rho$-full, $\rho(\cok(\phi\oplus\mu')) = \rho(\cok(\mu)) =0$ and
  thus $\phi \oplus \mu$ is $\rho$-full.

  Thus we have shown that for every finitely presented module $M$
  there exists a $\rho$-full map from ${^{\rho(M)}}R$  to $M$. Thus
  every object of $\fpmod{R}_{\Pi}$  is isomorphic to ${^{m}}[R]$ for
  some integer $m$. We also know that $\End_{\fpmod{R}_{\Pi}}([R]) =
  D$ a division ring and hence $\Hom_{\fpmod{R}_{\Pi}}([R], )$ defines
  an equivalence of categories between $\fpmod{R}_{\Pi}$ and $\mmod(D)
  $ as we set out to prove.

  Moreover the localisation functor is right exact and therefore the
  functor we now have via composition from $\fpmod{R}$ to $\mmod(D) $
  is right exact. By lemma \ref{lem:charrex}, it follows that it takes
  the form $\tensor_{R}D$ and so for every finitely presented module
  $M$, $\dim_D(M \tensor_{R}D) = \rho(M)$.
\end{proof}

Thus we have the proved the hard direction of Cohn's characterisation
of epimorphisms from a ring $R$ to division rings. For completeness we
state this theorem

\begin{theorem}
\label{Cohnepitodr}
Let $R$ be a ring. Then the epimorphisms from $R$ to division rings
are parametrised by the Sylvester rank functions $\rho$ on $\fpmod{R} $
such that $\rho(R)=1 $.

The parametrisation is given by associating to such an epimorphism
from $R$ to $D$ the Sylvester rank function given by $\rho(M)
=\dim_{D}M\tensor D$.

The inverse map from Sylvester rank functions such that $\rho(R)=1 $
to epimorphisms to a division ring is constructed as follows.
To a Sylvester rank function $\rho$, we associate the ring
homomorphism $\map{\phi}{R}{D}$ where $D$ is the endomorphism ring of
$[R]$ in $\fpmod{R}_{\sigma}$ where $\sigma$ is the severe right Ore
set of $\rho$-full maps be finitely presented modules over $R$. 
\end{theorem}

\section{Universal localisation via right Ore localisation} 
\label{sec:ulOl} 

At this point we return to our main study. In this section we begin by
showing that the category of finitely presented modules over the
universal localisation of $R$ at a set of maps between finitely
generated projective modules over $R$ may be obtained as the right Ore
localisation of $\fpmod{R} $ at the severe right Ore set $\sigma$
generated by $\Sigma$ in $\fpmod{R} $. We then use this to investigate
the $R$ module structure of modules in $\fpmod{R_{\Sigma}}$ by a
closer examination of the maps in $\sigma$.

Let $\Sigma$ be a set of maps between finitely generated projective
modules over the ring $R$ and let $\sigma$ be the severe right Ore set
generated by $\Sigma$. Then consider the functor
$\_\otimes_{R}R_{\Sigma} \colon \fpmod{R}\rightarrow
\fpmod{R_{\Sigma}}$. Since it inverts all elements of $\Sigma$ it must
invert $\sigma$ as well and so we obtain a new functor
$\Lambda_{\Sigma}\colon \fpmod{R}_{\sigma}\rightarrow
\fpmod{R_{\Sigma}}$. 
We shall use the notation $\downarrow^{\Sigma}$
for the restriction functor from $\Mod{R_{\Sigma}}$ to $\Mod{R}$.

\begin{theorem}
\label{th:rOlfp=fpul}
Let $\Sigma$ be a set of maps between finitely generated projective
modules over the ring $R$ and let $\sigma$ be the severe right Ore set
generated by $\Sigma$. Then the functor $\Lambda_{\Sigma}\colon
\fpmod{R}_{\sigma}\rightarrow \fpmod{R_{\Sigma}}$ is an equivalence
of categories. 
\end{theorem}
\begin{proof}
We begin by showing that $\Lambda_{\Sigma}{\downarrow}^{\Sigma}$ is
isomorphic to the realisation functor $\Gamma$. We need to show that
$M_{\sigma}\cong M\otimes_{R}R_{\Sigma}$. For each map $u$ in the
directed system whose limit is $M_{\sigma}$, $u\otimes_{R}R_{\Sigma}$
is an isomorphism and hence $M_{\sigma}\otimes_{R}R_{\Sigma} \cong
M\otimes_{R}R_{\Sigma}$. However, $M_{\sigma}$ is an $R_{\Sigma}$
module, because if $\alpha \colon P \rightarrow Q$ lies in $\Sigma$,
then $(\alpha,M_{\sigma})\colon \Hom(Q,M_{\sigma}) \rightarrow
\Hom(P,M_{\sigma})$ is the map $(\alpha,[M])\colon \Hom([Q],[M])
\rightarrow \Hom([P],[M])$ in $\fpmod{R}_{\sigma}$ after applying
the isomorphism of lemma \ref{th:limhom=hom}. But $\alpha$ is an
isomorphism in $\fpmod{R}_{\sigma}$ and so $(\alpha,[M])$ is an
isomorphism. Thus $M_{\sigma}$ is an $R_{\Sigma}$ module and since the
homomorphism from $R$ to $R_{\Sigma}$ is an epimorphism,
$M_{\sigma}\otimes_{R}R_{\Sigma} \cong M_{\sigma}$. At this stage, we
see that $\Lambda_{\Sigma}$ is full and faithful and its image lies in
$\fpmod{R_{\Sigma}}$; to be an equivalence we need that every
finitely presented module over $R_{\Sigma}$ is isomorphic to some
$\Lambda_{\Sigma}([M]) = M_{\sigma}$. However, $R_{\sigma}$ is
$R_{\Sigma}$ which is a projective object in the image of
$\downarrow^{\Sigma}$.  Since $\Gamma$ is full and faithful, $[R]$ is
a projective object of $\fpmod{R}_{\sigma}$, that is,
$\Gamma=\Hom([R],\ )$ is right exact. Since $\fpmod{R}_{\sigma}$ is
an additive category with cokernels, the image of $\Gamma$ and hence
$\Lambda_{\Sigma}$ is closed under cokernels and therefore every
finitely presented module over $R_{\Sigma}$ lies in the image of
$\Lambda_{\Sigma}$ as required.
\end{proof}

Whilst the proof of this theorem is still fresh in the mind of the
reader we note that we also proved the following corollary.

\begin{theorem}
\label{th:ulocind=dirlim}
Let $\Sigma$ be a set of maps between finitely generated projective
modules over $R$. Then for any finitely presented module $M$ over $R$,
$M\otimes_R{R_{\Sigma}} \cong M_{\sigma} = \varinjlim_{s\in
{_{M}\sigma}}M_{s}$ where ${_{M}\sigma}$ is the set of maps in the
severe right Ore set generated by $\Sigma$ that begin at $M$.
\end{theorem}

For this to be useful we need to understand the maps in $\sigma$ and
the next theorem gives us such a description. First we introduce some
relevant ideas. 

Let $\Sigma$ be a set of maps between finitely generated projective
modules over $R$; then when we invert $\Sigma$ we also invert other
maps between finitely generated projective modules. Thus if we invert
$\alpha$ and $\beta$, we also invert $\bigl(\begin{smallmatrix}
\alpha & h\\
0 & \beta
\end{smallmatrix}
\bigr)$ for any map $h$ between the correct projective modules. We say
that a set of maps $\Theta$ between finitely generated projective
modules is \define{upper triangularly closed} if for all
$\alpha,\beta\in\Theta$ and $h$ such that $\bigl(\begin{smallmatrix}
\alpha & h\\ 0 & \beta
\end{smallmatrix}
\bigr)$ is a map between finitely generated projective modules, then
this map lies in $\Theta$. The \define{upper triangular closure}
of $\Sigma$, $\overline{\Sigma}$, is the smallest upper triangularly closed
set of maps containing $\Sigma$. Exactly the same considerations apply
to define \define{lower triangularly closed} sets of maps and the
\define{lower triangular closure} $\underline{\Sigma}$ of a set
of maps $\Sigma$. Finally, we say that a set of maps is
\define{triangularly closed} if it is both upper and lower
triangularly closed and the \define{triangular closure} of a set
of maps $\Sigma$ is the smallest triangularly closed set of maps
$\tilde{\Sigma}$ containing it. It is clear that all maps in the
triangular closure of $\Sigma$ are inverted when we invert $\Sigma$
and so $R_{\Sigma}\cong R_{\overline{\Sigma}} \cong
R_{\underline{\Sigma}} \cong R_{\tilde{\Sigma}}$. It is a simple
matter to check that $\tilde{\Sigma}\subset
\sigma$ where $\sigma$ is the severe right Ore set generated by
$\Sigma$. It follows that the severe right Ore set generated by
$\tilde{\Sigma}$ or by either of $\underline{\Sigma}$ or
$\overline{\Sigma}$ is just $\sigma$. 

In the case where all the maps in $\Sigma$ are injective then so are
the maps in the triangular closure of $\Sigma$ and the maps in the
lower triangular closure of $\Sigma$ give presentations of all modules
in the extension closure of the modules $S_{\Sigma}$ as do the maps in
the upper triangular closure. 

We say that $t$ is a \define{good pushout} if there exists a map
$\tau$ in the lower triangular closure of $\Sigma$ such that $t$ is a
pushout of $\tau$. We say that $u$ is a \define{good surjection}
if there exists a diagram
\begin{equation*}
\xymatrix{
P \ar[r]^-{\tau}  & P' \ar[r]^-{a}  & N \ar[r]^-{u}  & N'  }
\end{equation*}
where $\tau$ is in the lower triangular closure of $\Sigma$, $\tau a
=0$ and $u$ is the cokernel of $a$. Of course, the conditions imply
that a good pushout or surjection lies in $\sigma$.  

In the case where all elements of $\Sigma$ are injective we shall see
later (or the reader can quickly check) that the set of good pushouts
is precisely the set of injective maps between finitely presented modules
whose cokernel lies in the extension closure of $S_{\Sigma}$.

\begin{theorem}
\label{th:sigmaisgpushgoodsurj}
Let $R$ be a ring and let $\Sigma$ be a set of maps between finitely
generated projective modules over $R$. Let $\sigma$ be the severe
right Ore set generated by $\Sigma$. Then if $s\in \sigma$ there
exists a good pushout $t$ and a good surjection $u$ such that $s=tu$.  
\end{theorem}
\begin{proof}
We prove that the set $\sigma'$ of maps of the form $tu$, where $t$ is
a good pushout and $u$ is a good surjection, is iself a severe right
Ore set and since it contains the generators of $\sigma$ and lies in
$\sigma$ it must be $\sigma$.

Clearly the set of good pushouts is closed under pushout. We show
firstly that it is closed under composition too. Consider the diagram
below where the left hand and right hand square are pushout diagrams
and $\alpha$ and $\beta$ are in the lower triangular closure of
$\Sigma$.  
\begin{equation*}
\xymatrix{
P \ar[r]^-{\alpha} \ar[d]_-{a}  & Q \ar[d]_<<{b}  & P_1%
\ar@{-->}[l]_-{e} \ar[r]^-{\beta} \ar@{-->}[lld]^<<<<<<{f} \ar[d]^-{c} &%
Q_1 \ar[d]^-{d} \\%
L \ar[r]_-{s}  & M \ar@<0.5ex>@{-}[r] \ar@<-0.5ex>@{-}[r]  &%
M \ar[r]_-{t}  & N  }%
\end{equation*}

So $s$ and $t$ are good pushouts. Since $\textmap{s}{b}$ is surjective
and $P_{1}$ is projective we can find maps $e,f$ such that
$c=eb+fs$. Now consider the map
\begin{equation*}
\begin{pmatrix}
a &\alpha &0\\
f &e &\beta
\end{pmatrix}\colon P\oplus P_{1}\rightarrow L\oplus Q \oplus Q_{1} 
\end{equation*}
and suppose that 
\begin{equation*}
\begin{pmatrix}
a &\alpha &0\\
f &e &\beta
\end{pmatrix}\begin{pmatrix}
x\\
y\\
z
\end{pmatrix} = 0
\end{equation*}
from which we see that $ax+\alpha y=0$ and so
$\textmap{x}{y}=\textmap{s}{b}w$ for some map $w$ since
$\textmap{s}{b}$ is the cokernel of $(a\ \alpha)$. Hence $cw +\beta z
= fsw+ebw+\beta z = 0$. Hence $\textmap{w}{z} = \textmap{t}{d}v$ for
some map $v$ since $\textmap{t}{d}$ is the cokernel of $(c\
\beta)$. Thus 
\begin{equation*}
\begin{pmatrix}
x\\
y\\
z
\end{pmatrix} =\begin{pmatrix}
st\\
bt\\
d
\end{pmatrix}v 
\end{equation*}
which shows that
\begin{equation*}
\begin{pmatrix}
st\\
bt\\
d
\end{pmatrix} \text{ is the cokernel of }\begin{pmatrix}
a &\alpha &0\\
f &e &\beta
\end{pmatrix} 
\end{equation*}
and hence $st$ is the pushout of $\textmap{\alpha\ 0}{e\ \beta}$ along
$a$ and hence $st$ is a good pushout. Thus the set of good pushouts is
closed under composition.

We show that the set of good surjections is closed under pushout.
Consider the diagram
\begin{equation*}
\xymatrix{
P \ar[r]^-{\alpha}  & Q \ar[r]^-{a}  & L \ar[r]^-{s} \ar[d]_-{b}  & %
M \ar[d]^-{b'} \\%
 &  & N \ar[r]_-{s'}  & K  }%
\end{equation*}
where $\alpha$ lies in the lower triangular closure of $\Sigma$,
$\alpha a=0$, $s$ is the cokernel of $a$ and the right hand square is
a pushout. Then $s'$ is the cokernel of $ab$ and consequently it too
is a good surjection. 

At this stage, we know that $\sigma'$ is closed under
pushouts. Consider the diagram
\begin{equation*}
\xymatrix{
K \ar[r]^-{s} \ar[d]_-{a}  & L \ar[r]^-{t} \ar[d]_-{b}  & M \ar[d]^-{c} \\
A \ar[r]_-{s'}  & B \ar[r]_-{'}  & C  }
\end{equation*}
where $s$ is a good pushout and $t$ is a good surjection and the two
squares are pushout diagrams. Then the outer rectangle is also a
pushout diagram and so the pushout of $st$ is $s't'$ where $s'$ is a
good pushout and $t'$ is a good surjection. 

Next we show that the set of good surjections is closed under
composition. Consider the diagram
\begin{equation*}
\xymatrix{
 & P_1 \ar[r]^-{\beta} \ar@{-->}[d]_-{d}  &%
Q_1 \ar@{-->}[d]_-{c} \ar[rd]^-{b} \\%
P \ar[r]_-{\alpha}  & Q \ar[r]_-{a}  & L \ar[r]_-{s}  & M \ar[d]^-{t} \\%
 &  &  & N  }%
\end{equation*}
where $\alpha,\beta$ lie in the lower triangular closure of $\Sigma$,
$\alpha a=0=\beta b$, $s$ is the cokernel of $a$ and $t$ is the
cokernel of $b$ and we need to construct and describe $c$ and
$d$. Since $s$ is surjective we pick $c$ so that $b=cs$. Then $\beta
cs=0$ and since $s$ is the cokernel of $a$ and $P_{1}$ is projective
there exists a map $d$ such that $da +\beta c =0$. Hence,
\begin{equation*}
\begin{pmatrix}
\alpha & 0\\
d & \beta
\end{pmatrix}\begin{pmatrix}
a\\
c
\end{pmatrix} = 0
\end{equation*}
and if $\textmap{a}{c}x=0$ then since $s$ is the cokernel of $a$,
$x=sy$ for some $y$ and then $by=csy=cx=0$ and so $y=tz$ and $x=stz$
which proves that $st$ is the cokernel of $\textmap{a}{c}$ and so the
set of good surjections is closed under composition.

Next we show that if $s$ is a good surjection and $t$ is a good
pushout such that $st$ exists, then there exist $s'$ and $t'$ where
$s'$ is a good surjection, $t'$ is a good pushout and $st=t's'$. Once
this is proved it is clear that $\sigma'$ is closed under
composition. 

Consider the diagram
\begin{equation*}
\xymatrix{
 & P \ar[r]^-{\alpha} \ar@{-->}[ld]^-{c} \ar[d]^-{a}  & Q \ar[d]^-{b} \\
K \ar[r]_-{s}  & M \ar[r]_-{t}  & N  }
\end{equation*}
where the square is a pushout diagram, $\alpha$ lies in the lower
triangular closure of $\Sigma$ and $s$ is a good surjection. Then we
can choose $c$ so that $a=cs$. So we form the diagram
\begin{equation*}
\xymatrix{
P \ar[r]^-{c} \ar[d]_-{\alpha}  & K \ar[r]^-{s} \ar[d]_-{t_1}  & M \ar[d]^-{t} \\
Q \ar[r]_-{d}  & L \ar[r]_-{s_1}  & N  }
\end{equation*}
where each square and the outer rectangle are pushout diagrams. But we
see that $st=-t_{1}s_{1}$ where $t_{1}$ is a good pushout and $s_{1}$
is a good surjection since it is a pushout of a good surjection. 

As stated above, this proves that $\sigma'$ is closed under
composition since if $s_{i}$ are good pushouts and $t_{i}$ are good
surjections for $i=1,2$ then $s_{1}t_{1}s_{2}t_{2} =s_{1}s't't_{2}$
where $s'$ is a good pushout and so is $s_{1}s'$ and $t'$ is a good
surjection and so is $t't_{2}$.

It remains to show that $\sigma'$ is severely right revengeful. 

Consider the diagram
\begin{equation*}
\xymatrix{
P \ar[r]^-{\alpha} \ar[d]_-{a}  & Q \ar[d]^-{b} \\
A \ar[r]_-{s}  & B \ar[r]_-{t}  & C \ar[r]_-{d}  & D \ar[r]_-{e}  & E  }
\end{equation*}
where the square is a pushout diagram, $\alpha$ lies in the lower
triangular closure of $\Sigma$, $t$ is a good surjection, $std=0$ and
$e$ is the cokernel of $d$. Assume for the moment that $b$ is
surjective. Then since $bt$ is surjective, $e$ is also the cokernel of
$btd$ and $\alpha btd=0$ so that $e$ is a good surjection and lies in
$\sigma'$. If $b$ is not surjective then replace the left hand square
by 
\begin{equation*}
\xymatrix{
P\oplus P' \ar[r]^-{\left(\begin{smallmatrix}
\alpha & 0\\
0 & I
\end{smallmatrix}\right)} \ar[d]_-{\left(\begin{smallmatrix}
a\\
p
\end{smallmatrix}\right)}  & Q\oplus P' \ar[d]^-{\left(\begin{smallmatrix}
b\\
ps
\end{smallmatrix}
\right)} \\
L \ar[r]_-{s}  & M  }
\end{equation*}
which is still a pushout diagram and now $\textmap{b}{ps}$ is
surjective. Thus, in all cases, $e$ lies in $\sigma'$ and $\sigma'$ is
a severe right Ore set which is what we set out to prove.
\end{proof}

If we do not assume that the maps in $\Sigma$, the set of matrices
between finitely generated projective modules, are injective then good
pushouts are relatively awkward to interpret; however, good
surjections are easy to understand. Let $T_{\alpha}$ be the cokernel
of $\alpha$. Then the cokernel of a map in the lower triangular
closure of $\Sigma$ is simply a module in the extension closure of the
set of modules $\{T_{\alpha}\}$ and every such module occurs as a
cokernel. Therefore good surjections are cokernels of maps from such
modules. Of course, it is clear that images of such modules in $M$
must lie in the kernel of the map from $M$ to
$M\otimes_{R}R_{\Sigma}$. Now suppose that all maps in $\Sigma$ are
injective; then good pushouts again are easy to describe. They are
simply injective maps whose cokernels lie in the extension closure of
the modules $T_{\alpha}$.

\begin{theorem}
\label{th:chargoodpush}
Let $\Sigma$ be a set of injective maps between finitely generated
projective modules over $R$. Let $S_{\Sigma}$ be the set of cokernels
of elements of $\Sigma$. Let $\cat{E}$ be the extension closure of
$S_{\Sigma}$. Then a map in the category of finitely presented modules
over $R$ is a good pushout if and only if it is injective and its
cokernel lies in $\cat{E}$. A map is a good surjection if and only if
it is surjective and its kernel is a factor of a module in $\cat{E}$.
\end{theorem}
\begin{proof}
The maps in the lower triangular closure of $\Sigma$ are
injective and their cokernels lie in $\cat{E}$. Therefore, their
pushouts have both these properties. 

Conversely, suppose that $f\colon M\rightarrow N$ is an injective map
whose cokernel is $T\in \cat{E}$. Choose a map $\alpha\colon
P\rightarrow Q$ in the lower triangular closure of $\Sigma$ whose
cokernel is $T$. Applying $\Hom(\ ,M)$ to the short exact sequence 
$\tses{P}{Q}{T}$, we see that the map from $\Hom(P,M)$ to $\Ext(T,M)$
is surjective and therefore there exists a map $j \colon P \rightarrow
M$ such that the pushout of $\alpha$ along $j$ is $f$ as required.

Now let $s \colon M \rightarrow N$ be a good surjection. So there
exists a map $\alpha\colon P\rightarrow Q$ in the lower triangular
closure of $\Sigma$, a map $a \colon Q \rightarrow M$, such that
$\alpha a = 0$ and $s$ is the cokernel of $a$. Then the map $a$
induces a map $b \colon \cok\alpha \rightarrow M$ with the same image
as $a$ and so $s$ is also the cokernel of $b$.
\end{proof}

We are in a position now to use our description of the severe right
Ore set generated by a set of maps between finitely generated
projective modules over $R$ as the compositions of good pushouts and
good surjections to give module-theoretic information about the
induction functor $\tensor_{R} R_{\Sigma}$. Firstly, we should like to
understand the kernel of the homomorphism from a module $M$ to $M
\tensor_{R} R_{\Sigma}$ and in particular to understand those modules
such that $M \tensor_{R} R_{\Sigma}=0$. Secondly, we should like to
understand precisely which maps between finitely generated projective
modules over $R$ are inverted by the ring homomorphism from $R$ to
$R_{\Sigma}$. 

\section{Induction}
\label{sec:ind}
 
In this section we shall assume that the maps in
$\Sigma$ are all injective. The effect of this is that we can and
should replace consideration of $\Sigma$ by the set of modules that
are the cokernels of the elements of $\Sigma$. For if $\Sigma'$ is
some different set of injective maps between finitely generated
projective modules over $R$ having the same set of cokernels (up to
isomorphism) as $\Sigma$ then clearly $R_{\Sigma}\cong R_{\Sigma'}$. 
Thus given a set $S$ of finitely presented modules of homological
dimension at most $1$ where $\cat{S} $ is the full subcategory of $
\fpmod{R}$ whose objects are isomorphic to modules in $S$, we define
$R_{\cat{S}} \cong R_{S}$ to be $R_{\Sigma}$ where $\Sigma$ is some
set of injective maps between finitely generated projective modules
whose set of cokernels is $S$.  

In the discussion of $R_{S}$, we need the notion of a torsion theory.
See section 5.1 of \cite{CohnFIR} for a brief summary.  Given any set
of modules $S$, we have an associated torsion theory generated by $S$;
a module $F$ is torsion-free if $\Hom(N,F)=0$ for all modules $N\in S$
and $T$ is torsion if $\Hom(T,F)=0$ for every torsion-free module. The
torsion theory is usually thought of as a pair
$\mathcal{T},\mathcal{F}$ where $\mathcal{T}$ is the class of torsion
modules and $\mathcal{F}$ is the class of torsion-free modules.  In
the best cases, one would hope that the torsion submodule of $M$ with
respect to the torsion theory generated by $S$ would be the kernel of
the homomorphism from $M$ to $M\tensor R_{S}$.

If $S$ is a set of
finitely presented modules of homological dimension at most $1$, then
the torsion modules have a useful description.

\begin{lemma}
\label{th:torsthgenbyfphd1}
Let $S$ be a set of finitely presented modules of homological
dimension at most $1$. Let $\mathcal{T},\mathcal{F}$ be the torsion
theory generated by $S$. Let $\cat{U}$ be the full subcategory of
modules that are factor of modules in $\cat{E}$, the extension closure
of $S$.  Then every finitely generated module in $\mathcal{T}$ is
isomorphic to some module in $\cat{U}$. If $T$ is a module in
$\mathcal{T}$, then every finitely generated submodule of $T$ lies in
a larger submodule isomorphic to a module in $\cat{U}$.
\end{lemma}
\begin{proof}
We show first that $\cat{U}$ is closed under extensions. For if we
have surjections $s_{i} \colon E_{i} \rightarrow U_{i}$ for $i=1,2$
where $E_{i}\in \cat{E}$ and a short exact sequence
$\tses{U_{1}}{U}{U_{2}}$ then we form the pullback 
\begin{equation*}
\pullses{U_{1}}{X}{E_{2}}{U}{U_{2}} 
\end{equation*}
which gives a surjection from $X$ to $U$ where $X$ is an extension of
$E_{2}$ on $U_{1}$. Since $E_{2}$ has homological dimension at most
$1$, the map from $\Ext(E_{2},E_{1})$ to $\Ext(E_{2},U_{1})$ is
surjective and consequently there exists an extension of $E_{2}$ on
$E_{1}$, $E$ with a surjection onto $X$ and hence onto $U$. Thus
$\cat{U}$ is closed under extensions. 

Suppose that $T$ is a torsion module. Then every nonzero factor of $T$
has a nonzero map from some module in $S$. Let $T_{1}$ and $T_{2}$
be submodules of $T$ such that for every finitely generated submodule
$N$ of $T_{i}$, there exists a module $U\in \cat{U}$ where $N\subset
U\subset T_{i}$, then the same holds for $T_{1}\oplus T_{2}$ and hence
for $T_{1}+T_{2}$ since $\cat{U}$ is closed under factors. Also a
directed union of such submodules is again such a module. So there
exists a maximal such submodule, call it $T'$. Assume that $T'\neq
T$. Let $\phi \colon E \rightarrow T/T'$ be a nonzero map where $E\in
S$. Choose a map between finitely generated projective modules
$\alpha\colon P\rightarrow Q$ such that $\cok\alpha=E$. Choose a map
$\beta\colon Q\rightarrow T$ that lifts $\phi$. Then the image of
$\alpha\beta$ must lie in $T'$ and therefore there exists a module
$U\subset T'$ where $U\in \cat{U}$ that contains the image of
$\alpha\beta$. Let $\gamma\colon P\rightarrow U$ be the induced
map. Consider the submodule $U'=U+\im\beta$. Then the image of $U'$ in
$T/T'$ is $\im\phi$. Also $U\cap \im\beta$ contains $\im
\alpha\beta$. It follows that $U'$ is a factor of the module $U_{1}$,
the pushout of $Q$ along the map $\gamma$ in the diagram below
\begin{equation*}
\pushses{P}{Q}{E}{U}{U_{1}} 
\end{equation*}
which shows also that $U_{1}\in \cat{U}$. Therefore, $U'\in
\cat{U}$. But any module in $\cat{U}$ certainly satisfies the
condition that every finitely generated submodule lies in a submodule
lying in $\cat{U}$; therefore, because the sum of two such submodules
is also a module of this type so is $T'+U'$ which contradicts the
maximality of $T'$. This contradiction implies that $T'=T$ and proves
the second part of the theorem. The first follows at once.
\end{proof}

Let $\Sigma$ be a set of injective maps between finitely generated
projective modules over $R$ and let $S$ be the set of cokernels of
elements of $\Sigma$; so $R_{S}=R_{\Sigma}$. It is worth pointing out
that if $M$ is torsion-free with respect to the torsion theory
generated by $S$ then it does not follow that $M$ embeds in
$M\otimes_{R}R_{S}$ since there may be a short exact sequence
$\tses{M}{N}{T_{1}}$ where $T_{1}$ lies in the extension closure of
$S$ and a homomorphism from $T_{2}$ in the extension closure of $S$ to
$N$ whose image intersects $M$; this intersection must then lie in the
kernel of the homomorphism from $M$ to $M\otimes_{R}R_{S}$. There are
conditions which make sure that this does not happen and we shall be
considering one such later in this paper but first we introduce a
related problem.

We should like to be able to describe the complete set of maps between
finitely generated projective modules that become invertible under the
ring homomorphism from $R$ to $R_{\Sigma}$. This is equivalent to
describing the set of finitely presented modules $\{M\}$ of
homological dimension at most $1$ over $R$ such
that $M\otimes_{R}R_{\Sigma}= 0= \Tor_{1}^{R}(M,R_{\Sigma})$ by
considering their presentations. In fact, we have another way to
recognise such modules.

\begin{theorem}
\label{th:invmap=invHom}
Let $R$ be a ring and let $\Sigma$ be a set of injective maps between
finitely generated projective modules over $R$. Let $M$ be a finitely
presented module of homological dimension at most $1$. Then $\Hom(M,\
) $ and $\Ext(M,\ ) $ vanish on $R_{\Sigma}$ modules if and only if $M
\tensor R_{\Sigma} = 0 = \Tor_{1}^{R}(M, R_{\Sigma}) $.
\end{theorem}
\begin{proof}
  Of course, $\Hom(M,\ ) $ vanishes on $R_{\Sigma}$ modules if and
  only if $M \tensor R_{\Sigma}=0$. Let $\tses{P}{Q}{M} $ be some
  presentation of $M$ as $R$ module where $P, Q$ are finitely
  generated projective modules. Applying $\tensor R_{\Sigma}$ gives an
  exact sequence $P \tensor R_{\Sigma}\to Q \tensor R_{\Sigma} \to M
  \tensor R_{\Sigma}= 0$ so that the first map must be split
  surjective and its kernel is $\Tor_{1}^{R}(M, R_{\Sigma})$. On the
  other hand, applying $\Hom(\ ,X) $ for some $R_{\Sigma}$ module
  gives the exact sequence $\Hom(Q,X)\to \Hom(P,X)\to \Ext(M,X)\to 0 $
  and since $\Hom(K,\ ) = \Hom(K \tensor R_{\Sigma},\ ) $ on
  $R_{\Sigma}$ modules for any $R$ module we see that $\Ext(M,\ )
  \cong \Hom(\Tor_{1}^{R}(M, R_{\Sigma}),\ ) $
  vanishes on $R_{\Sigma}$ modules if and only if $\Tor_{1}^{R}(M,
  R_{\Sigma})=0$. 
\end{proof}

Given a set of injective maps between finitely generated projective
modules $\Sigma$, we shall use the notation $\cat{S}(\Sigma) $ for the
full subcategory of $\fpmod{R} $ of modules $M$ of homological
dimension at most $1$ such that $R_{\Sigma}$ inverts their
presentations or equivalently $ M \tensor R_{\Sigma} =0=
\Tor_{1}^{R}(M, R_{\Sigma})$. We call this the category of
$R_{\Sigma}$ trivial modules. Clearly $R_{\cat{S}(\Sigma)} =
R_{\Sigma}$. We want to describe some obvious closure conditions for
this category.

The last closure condition of the following lemma is at first sight a
little odd; however it will turn out to be useful to us later.

\begin{lemma}
\label{lem:closureSigmatriv}
Let $\Sigma$ be a set of injective maps between finitely generated
projective modules over the ring $R$. Then $\cat{S}(\Sigma)$ is closed
under extensions and closed under kernels of surjective maps. It is
also closed under cokernels of injective maps whose cokernel has
homological dimension $1$.

Finally, if $\map{\phi} {A}{B} $ is a map in $\cat{S}(\Sigma)$ whose
cokernel has homological dimension $1$ then $\cok\phi$ lies in
$\cat{S}(\Sigma)$ and $\im\phi$ and $\ker\phi$ also lie in
$\cat{S}(\Sigma)$ whenever they are finitely presented.
\end{lemma}
\begin{proof}
Extensions of modules of homological dimension at most $1$ and kernels
of surjective maps between modules of homological dimension at most
$1$ must have homological dimension at most $1$ and therefore applying
$\Hom(\ ,X) $ for any $R_{\Sigma}$  module $X$ to a relevant short
exact sequence shows each of the closure conditions in the first
paragraph.

From the short exact sequence $\tses{\im\phi}{B}{\cok\phi}$, it
follows that $\im\phi$ has homological dimension at most $1$. Applying
$\Hom(\ ,X) $ for $X$ an $R_{\Sigma}$ module to this short exact
sequence shows that $\Ext(\im\phi,X) =0$ for every such $X$ and
applying $\Hom(\ ,X)$ to the short exact sequence
$\tses{\ker\phi}{A}{\im\phi}$ shows that $\Hom(\im\phi,X) =0$ for
every such $X$; so $\im\phi\in \cat{S}(\Sigma)$ whenever $\im\phi$ is
finitely presented. Looking at the short exact sequence
$\tses{\im\phi}{B}{\cok\phi}$ again, we see that $\cok\phi$ is
finitely presented and satisfies $\Hom(\cok\phi,X) = 0 = \Ext(
\cok\phi,X) $ for every $R_{\Sigma}$ module $X$ so that $\cok\phi \in
\cat{S}(\Sigma)$. Finally, we show by reconsidering the short exact
sequence $\tses{\ker\phi}{A}{\im\phi}$ that $\Hom(\ker\phi,X) = 0 =
\Ext( \ker\phi,X)$   for every $R_{\Sigma}$ module $X$ and $\ker\phi$
has homological dimension at most $1$ so if it is finitely presented
then it too must lie in $\cat{S}(\Sigma)$.
\end{proof}

We shall say that a full subcategory of $\fpmod{R} $ whose objects are
modules of homological dimension at most $1$ satisfying the the
closure conditions in this lemma a pre-localising subcategory.

We begin with a characterisation for arbitrary universal localisations
of the kernel of the natural map from a finitely presented module to
the induced module and of when a finitely presented module becomes the
zero module under induction. 

\begin{theorem}
\label{th:kerindgen}
Let $\cat{E} $ be a subcategory of $\fpmod{R} $ closed under
extensions whose objects have homological dimension at most $1$. Let
$M$ be a finitely presented module. Then $m\in M$ lies in the kernel
of the homomorphism from $M$ to $M\tensor R_{\cat{E}}$ if and only if there
exist a short exact sequence $\tses{M}{N}{E}$ and a homomorphism
$\map{\phi}{E'}{N}$ where $E,E'\in \cat{E}$ and $m$ lies in the image
of $\phi$.

Further, $M\tensor R_{\cat{E}} =0$ if and only if there
exist a short exact sequence $\tses{M}{N}{E}$ and a homomorphism
$\map{\phi}{E'}{N}$ where $E,E'\in \cat{E}$ and $M$ lies in the image
of $\phi$.
\end{theorem}
\begin{proof}
The functor $\tensor R_{\cat{E}}$ is equivalent to the right Ore
localisation at the  severe right Ore set $\Pi$ generated by a set of
presentations of the modules in $\cat{E}$. So $m$ lies in the kernel
of the homomorphism from $M$ to $M\tensor R_{\cat{E}}$ if and only if
the map $\map{l_{m}}{R}{M}$ given by $l_{m}(r) =mr$ becomes the zero
map over $R_{\cat{E}}$ which holds if and only if there exists a map $s$ 
in $\Pi$ such that $l_ms=0$.

However, $s=iu$ where $i$ is a good pushout and $u$ is a good
surjection. The result follows at once from our description of good
pushouts and good surjections in theorem \ref{th:chargoodpush} in the
previous section.

The second result follows by taking the identity map on $M$ which must
become the zero map over $R_{\cat{E}}$ if and only if $M\tensor
R_{\cat{E}}=0$.  
\end{proof}

We do not expect to have a better theorem than this for arbitrary
universal localisations; however, there is a relatively common
situation where we can prove a better theorem.

\begin{theorem}
\label{th:keristorsion}
Let $\cat{E}$ be a pre-localising subcategory of
$\fpmod{R} $. Assume that the kernel of any map in $\cat{E}$ is a
torsion module with respect to the torsion theory generated by
$\cat{E}$. 
Then the kernel of
the homomorphism from a finitely presented module $M$ to $M\tensor
R_{\cat{E}}$ is the torsion submodule of $M$ with respect to the
torsion theory generated by $\cat{E}$.

In particular, if $M $ is torsion-free with respect to this torsion
theory then $M$ is an $R$-submodule of $M\tensor R_{\cat{E}}$.
\end{theorem}
\begin{proof}
  Using the last theorem, we see that $m\in M$ lies in the kernel of
  the homomorphism from $M$ to $M\tensor R_{\cat{E}}$ if and only if
  there exist a short exact sequence $\tses{M}{N}{E}$ and a map
  $\map{\phi}{E'}{N} $ where $E,E'\in \cat{E}$ such that the image of
  $\phi$ contains $m$. We consider the induced map from $E'$ to
  $E$. By assumption, the kernel of this map $K$ ,is torsion with
  respect to the torsion theory generated by $\cat{E}$ and the image
  of $K$ in $N$ lies in $M$, is torsion and must contain $m$. The
  result follows.
\end{proof}

There are a couple of other ways in which the condition that the
kernel of a map in a pre-localising category should be torsion with
respect to the torsion theory generated by the pre-localising category
is decisive for us. The next two theorems show that this condition
forces the category of $R_{\cat{S}}$-trivial modules to be just
$\cat{S}$. Then we show that the ring homomorphism from $R$ to
$R_{\cat{S}}$ is stably flat which is the condition for the derived
category of $R_{\cat{S}}$ modules to be the derived category of $R$
modules which in turn implies a localisation sequence in algebraic
$K$-theory. 

\begin{theorem}
\label{th:StrivisS}
Let $\cat{S}$ be a pre-localising subcategory
of $\fpmod{R} $. Assume that the kernel of any map in $\cat{E}$ is a
torsion module with respect to the torsion theory generated by
$\cat{E}$. 
Let $\Sigma$ be a set of presentations of all the modules
in $\cat{S}$. Then $\cat{S}= \cat{S}(\Sigma)$. 
\end{theorem}
\begin{proof}
  Let $M\in \cat{S}(\Sigma)$. Then by the previous theorem, $M$ is a
  torsion module with respect to the torsion theory generated by
  $\cat{S}$. Since $\cat{S}$ consists of modules whose homological
  dimension is at most $1$, there exists a short exact sequence
  $\tses{N}{T}{M}$ where $T\in \cat{S}$. Since $M\in \cat{S}(\Sigma)$,
  $N \tensor R_{\Sigma} \cong\Tor_{1}^{R}(M,R_{\Sigma}) = 0$ and so
  $N$ is also a torsion module with respect to the torsion theory
  generated by $\cat{S}$ and so there exists a surjective map from
  some $T'\in \cat{S}$ to $N$ which gives us a map from $T'$ to $T$
  with cokernel $M$ which is a module of homological dimension at most
  $1$. By the definition of a pre-localising subcategory of
  $\fpmod{R} $, $M\in \cat{S}$.
\end{proof}
 
We recall that Ranicki and Neeman introduced the notion of a stably
flat ring extension. They say that $S$ is a stably flat $R$ ring if
and only if $\Tor^{R}_{i}(S,S)=0 $ for all $i$. They show that there
is a long exact sequence in algebraic $K$-theory for universal
localisation when $R_{\Sigma}$  is a stably flat $R$ ring. We note
that our condition on a pre-localising category implies that the
universal localisation is often stably flat.

\begin{theorem}
\label{th:locisstflat}
Let $\cat{E}$ be a pre-localising subcategory of $\fpmod{R} $ and
assume that kernel of maps in $\cat{E}$ are torsion. Then
$R_{\cat{E}}$ is a stably flat $R$ ring. 
\end{theorem}
\begin{proof}
  Let $\tau$ be the set of good pushouts whose domain is $R$. Given
  $t\in \tau$, we call the codomain of $t$, $M_{t}$. Note that $\tau$
  is a directed system and we let $\hat{R}= \varinjlim_{t\in\tau}
  M_{t}$. Then the map from $\hat{R}$ to $R_{\cat{E}}= \hat{R}\tensor
  R_{\cat{E}}$ is surjective and its kernel $K$ is torsion with
  respect to the torsion theory generated by $\cat{E}$.

  It is clear that $\Tor_{i}^{R}(\hat{R}, R_{\cat{E}}) =0$ since
  $\hat{R}$ is a direct limit of the modules $M_{t}$ and each of these
  is an extension of $R$ by a module in $\cat{E}$ and $\Tor_{i}^{R}(E,
  R_{\cat{E}}) =0 $ for every $E\in \cat{E}$. Thus it is enough to
  show that $ \Tor_{i}^{R}(K, R_{\cat{E}}) =0$. Regarding $K$ as the
  direct limit of its finitely generated torsion submodules, it is
  enough to show that $\Tor_{i}^{R}(\ , R_{\cat{E}}) $ vanishes on
  finitely generated modules torsion with respect to the torsion
  theory generated by $\cat{E}$ that are submodules of the modules
  $M_{t}$. We begin by showing that $\Tor_{i}^{R}(\ , R_{\cat{E}}) $
  vanishes on finitely generated modules torsion with respect to the
  torsion theory generated by $\cat{E}$ that are submodules of the
  modules in $\cat{E}$.  

  Let $T\subset F$ be such a module (where $F\in \cat{E}$ ) and choose
  some short exact sequence $\tses{L}{E}{T}$ where $E\in \cat{E}$. $L$
  must be torsion because it in the kernel of the map from $E$ to $F$.
  Firstly, $\Tor_{1}^{R}(T, R_{\cat{E}}) = L \tensor R_{\cat{E}}=0$
  and so $\Tor_{1}^{R}(\ , R_{\cat{E}}) $ vanishes on arbitrary
  torsion submodules of modules in $\cat{E}$. Secondly, $\Tor_{i+1}(T,
  R_{\cat{E}}) \cong \Tor_{i}(L, R_{\cat{E}}) $ and so the assumption
  that $\Tor_{i}(\ , R_{\cat{E}}) $ vanishes on arbitrary torsion
  submodules of modules in $\cat{E}$ implies that $\Tor_{i+1}(\ ,
  R_{\cat{E}})$ also vanishes on all finitely generated torsion
  submodules of modules in $\cat{E}$ and hence vanishes on arbitrary
  torsion submodules of modules in $\cat{E}$. Thus we are done by
  induction.  

  Now let $U$ be a finitely generated torsion submodule of some
  $M_{t}$. We choose some short exact sequence $\tses{L}{E}{U}$ where
  $E\in \cat{E}$.  Consider the short exact sequence
  $\tses{R}{M_{t}}{E_{t}}$ where $E_{t}\in \cat{E}$. The kernel of the
  induced map from $E$ to $E_{t}$ is a torsion module $K\supset L$ and
  $K/L\subset R$ from which it follows that $K=L$ and $U$ is a
  submodule of $E_{t}$ and we have shown in the previous paragraph
  that $\Tor_{i}(U, R_{\cat{E}})=0$. It follows that $\Tor_{i}(K,
  R_{\cat{E}})=0$ and hence from the short exact sequence
  $\tses{K}{\hat{R}}{R_{\cat{E}}}$ that $\Tor_{i}^{R}(R_{\cat{E}},
  R_{\cat{E}}) =0$ for all $i>0$.  
\end{proof}

Finally we state what our localisation sequence in algebraic
$K$-theory is. 

\begin{theorem}
\label{th:localgKth}
Let $\cat{E}$ be a pre-localising subcategory of $\fpmod{R} $ and
assume that kernels of maps in $\cat{E}$ are torsion. Then there is a
long exact sequence in algebraic $K$-theory 
\begin{equation*}
\dots K_{i}(\cat{E})\to K_{i}(R)\to K_{i}(R_{\cat{E}})\to \dots
K_{0}(R)\to K_{0}(R_{\cat{E}})     
\end{equation*}
\end{theorem}
\begin{proof}
This follows at once from the previous two theorems and
\cite{RanNeeman}. 
\end{proof}

\end{document}